\documentclass[a4paper, 12pt]{article}
\pdfoutput=1

\usepackage[T1]{fontenc}
\usepackage[utf8]{inputenc}

\usepackage{amsfonts,amsmath,amssymb,mathtools,dsfont,microtype} 
\usepackage[usenames, dvipsnames]{xcolor} 
\usepackage[noBBpl]{mathpazo} 

\DeclareFontFamily{U} {cmr}{}

\DeclareFontShape{U}{cmr}{m}{n}{
	<-6> cmr5
	<6-7> cmr6
	<7-8> cmr7
	<8-9> cmr8
	<9-10> cmr9
	<10-12> cmr10
	<12-> cmr12}{}

\DeclareSymbolFont{Xcmr} {U} {cmr}{m}{n}
\DeclareMathSymbol{\Delta}{\mathord}{Xcmr}{'001}
\DeclareMathSymbol{\Upsilon}{\mathord}{Xcmr}{'007}
\DeclareMathSymbol{\Omega}{\mathord}{Xcmr}{'012}

\usepackage[labelsep = period, labelfont = bf, justification = centering]{caption}
\usepackage{float,graphicx,subcaption}
\usepackage{booktabs,makecell}
\usepackage{enumitem,mdwlist}
\setlist[itemize]{topsep=0ex,itemsep=0ex,parsep=0.4ex}
\setlist[enumerate]{topsep=0ex,itemsep=0ex,parsep=0.4ex}

\usepackage[left = 2.5cm, right = 2.5cm, top = 2.5cm, bottom = 2.5cm, headsep = 12pt, headheight = 15pt]{geometry}
\usepackage{parskip}

\usepackage[hyphens]{url} 
\usepackage[linktoc = all, hidelinks, colorlinks, unicode=true, pagebackref=true]{hyperref} 
\usepackage[capitalise, compress, nameinlink, noabbrev]{cleveref} 
\hypersetup{linkcolor={blue!70!black}, citecolor={green!70!black}, urlcolor={blue!70!black}}

\usepackage{amsthm,thmtools,thm-restate}
\declaretheorem[name = Theorem, numberwithin = section, style = plain]{theorem}

\declaretheorem[name = Corollary, numberlike = theorem, style = plain]{corollary}
\declaretheorem[name = Conjecture, numberlike = theorem, style = plain]{conjecture}
\declaretheorem[name = Definition, numberlike = theorem, style = definition]{definition}
\declaretheorem[name = Lemma, numberlike = theorem, style = plain]{lemma}

\DeclareFontFamily{U}{matha}{\hyphenchar\font45}
\DeclareFontShape{U}{matha}{m}{n}{
	<5> <6> <7> <8> <9> <10> gen * matha
	<10.95> matha10 <12> <14.4> <17.28> <20.74> <24.88> matha12
}{}
\DeclareSymbolFont{matha}{U}{matha}{m}{n}

\DeclareMathSymbol{\specialuparrow}{\mathrel}{matha}{"D2}
\DeclareMathSymbol{\specialrightarrow}{\mathrel}{matha}{"D1}

\renewcommand*{\backref}[1]{}
\renewcommand*{\backrefalt}[4]{
	\ifcase #1 Not cited.%
	\or $\specialuparrow$#2%
	\else $\specialuparrow$#2%
	\fi%
}

\renewcommand{\epsilon}{\varepsilon}
\renewcommand{\ge}{\geqslant}
\renewcommand{\le}{\leqslant}
\renewcommand{\geq}{\geqslant}
\renewcommand{\leq}{\leqslant}

\DeclarePairedDelimiter{\abs}{\lvert}{\rvert}

\DeclarePairedDelimiter{\floor}{\lfloor}{\rfloor}
\DeclarePairedDelimiter{\set}{\{}{\}}

\DeclareMathOperator{\diam}{diam}
\newcommand{\defn}[1]{\textcolor{Maroon}{\emph{#1}}}

\newcommand{\bE}{\mathbb{E}}
\newcommand{\bN}{\mathbb{N}}
\newcommand{\bP}{\mathbb{P}}
\newcommand{\bR}{\mathbb{R}}
\newcommand{\cO}{\mathcal{O}}
\newcommand{\cP}{\mathcal{P}}
\newcommand{\cG}{\mathcal{G}}

\title{Reconstructing a point set from a random subset of its pairwise distances}
\author{}
\date{}

\begin{document}

\author{Ant\'{o}nio Gir\~{a}o\footnotemark[2]
\qquad Freddie Illingworth\footnotemark[2]
\qquad Lukas Michel\footnotemark[2] \\
Emil Powierski\footnotemark[2]
\qquad Alex Scott\footnotemark[2]}

\maketitle

\begin{abstract}
	Let $V$ be a set of $n$ points on the real line. Suppose that each pairwise distance is known independently with probability $p$. How much of $V$ can be reconstructed up to isometry?

    We prove that $p = (\log n)/n$ is a sharp threshold for reconstructing all of $V$ which improves a result of Benjamini and Tzalik. This follows from a hitting time result for the random process where the pairwise distances are revealed one-by-one uniformly at random. We also show that $1/n$ is a weak threshold for reconstructing a linear proportion of $V$.
\end{abstract}

\renewcommand{\thefootnote}{\fnsymbol{footnote}} 


\footnotetext[2]{Mathematical Institute, University of Oxford, United Kingdom (\textsf{\{girao,illingworth,michel,powierski,\allowbreak scott\}@maths.ox.ac.uk}). Research of AG, FI, and AS supported by EPSRC grant EP/V007327/1.}

\renewcommand{\thefootnote}{\arabic{footnote}} 

\section{Introduction}

Let $V \subset \bR^1$ be a finite set of points on the real line and suppose that all we know about the points are the distances between some pairs $\cP \subset \binom{V}{2}$ of them. More precisely, labelling the points $v_1, \dotsc, v_n$, if the pair $v_i v_j \in \cP$, then the distance between $v_i$ and $v_j$ is given\footnote{The problem where the points are not identified and only the multiset of pairwise distances is given has also been studied, see \cite{ACKR1989,PRS2003}. This is not in general enough to reconstruct a point set, as $A+B$ and $A-B$ are not distinguishable.  However, sets in at most two dimensions are reconstructible from subsets of constant size.}. How much of $V$ can be reconstructed? Can all of $V$ be reconstructed? \defn{Reconstructing} a set $U \subseteq V$ means deducing the positions of the labelled points in $U$ up to isometry or, equivalently, deducing all pairwise distances of points in $U$. In this paper we consider the case where each pairwise distance is known independently with probability $p$, that is, $(V, \cP)$ is distributed as the Erd\H{o}s-Renyi binomial random graph $\cG(n, p)$.

Benjamini and Tzalik~\cite{BT2022} recently proved the following result.

\begin{theorem}[Benjamini and Tzalik~\cite{BT2022}]\label{thm:BTrand}
    Let $V$ be a set of $n$ points on the real line. There is a sufficiently large constant $C$ such that if the graph $G$ of known pairwise distances $(V, \cP)$ is distributed as $\cG(n, p)$ where $p = C \log(n)/n$, then the whole of $V$ can be reconstructed with high probability (whp).
\end{theorem}

We strengthen this result. Our first result identifies two important landmarks in the evolution of the largest reconstructible set as $p$ increases.
We improve on \cref{thm:BTrand} by proving a sharp threshold result for reconstructing the whole of $V$, as well as identifying the threshold for reconstructing a positive fraction of $V$.
\begin{theorem}\label{thm:linear}
    Let $V$ be a set of $n$ points on the real line. Suppose the graph $G$ of known pairwise distances $(V, \cP)$ is distributed as $\cG(n, p)$. Then the following hold whp.
    \begin{enumerate}[label = \alph{*}., ref = \alph{*}]
        \item If $p \geq 42/n$, then there is a reconstructible set of size $\Omega(n)$. \label{part:linear}
        \item  If $p n \to \infty$, then there is a reconstructible set of size $(1 - o(1))n$. \label{part:almostall}
        \item If $p \leq \frac{\log n + \log\log n - \omega(1)}{n}$, then it is not possible to reconstruct the whole of $V$. \label{part:notall}
        \item If $p \geq \frac{\log n + \log\log n + \omega(1)}{n}$, then the whole of $V$ can be reconstructed. \label{part:all}
    \end{enumerate}
\end{theorem}
These results are best possible up to the constant factor in part~\ref{part:linear}\footnote{We have not optimised the value 42 in part~\ref{part:linear}: a short calculation shows the best our methods could give is 9, while we believe the correct threshold is $1 + \epsilon$ (see \cref{sec:openproblems} for further discussion).}. Indeed, when $p = c/n$ for constant $c < 1$, the largest connected component in $G$ has size $\cO(\log n)$. When $p = c/n$ for constant $c > 1$, whp the largest component of $G$ is a giant component of size at most $dn$ for some constant $d < 1$.  

The final two parts of \cref{thm:linear} follow from a stronger hitting time result. Suppose that the pairwise distances are revealed one-by-one in a random order: in other words, the graph of known pairwise distances follows the random graph process $(G_t \colon 0 \le t \le \binom{n}{2})$. We will prove that whp $V$ is reconstructible when $G_t$ has minimum degree at least two. However, $V$ may also be reconstructible slightly before this. If some point $u$ is incident to exactly one revealed distance, then $u$ has two possible positions relative to its neighbour but it may be the case that one of these possible positions is already occupied and so in fact $u$ can be reconstructed. To this end, we say an ordered pair of distinct points $(u, v)$ is \defn{secure} if $w = 2v - u$, the reflection of $u$ over $v$, is also in $V$.

\begin{theorem}\label{thm:hitting}
    Let $V$ be a set of $n$ points on the real line. Suppose the distances between pairs of points in $V$ are revealed one-by-one in a uniformly random order. Then whp $V$ is reconstructible exactly at the first time that both the following hold.
    \begin{itemize}
        \item Every point is incident to at least one revealed distance.
        \item If a point $u$ is incident to exactly one revealed distance, which is to point $v$, then $(u, v)$ is secure.
    \end{itemize}
\end{theorem}

It follows that if every point is incident to at least two revealed distances, then whp the whole of $V$ is reconstructible. In fact, this moment is whp the hitting time for full reconstructibility if and only if there are $o(n^2)$ secure pairs in $V$. This is because the unique distance incident to the final point with only one revealed distance is uniformly random amongst $\binom{V}{2}$. An example of a set with $\Omega(n^2)$ secure pairs is $V = \set{1, 2, \dotsc, n}$. In the other direction, we will show (\cref{thm:arbitrarily}) that if there are $\omega(1)$ points incident to only one revealed distance, then whp the whole of $V$ is not reconstructible. 

It is also possible to obtain an algorithmic hitting time result that does not need knowledge of the underlying point set $V$ to determine the point at which the entirety of $V$ is reconstructible.

\begin{theorem}\label{thm:hittingalg}
    There is an online algorithm with polynomial expected running time in $n$ that takes the revealed distances one-by-one and whp recognises the first time at which $V$ is reconstructible and reconstructs it.
\end{theorem}

Reconstructibility is strongly related to graph rigidity, which is concerned with generic embeddings of graphs in $\bR^d$. An embedding of a set of vertices $V$ in $\bR^d$ is \defn{generic} if the set of $d \abs{V}$ coordinates of the vertices is algebraically independent over the rationals. 
A graph is \defn{globally rigid in $\bR^d$} if it has some generic embedding in $\bR^d$ which is reconstructible from its edge lengths.
In fact, a graph is globally rigid if and only if all of its generic embeddings are reconstructible from their edge lengths~\cite{Connelly2005,GHT2010}.

The (global) rigidity of the random graph in $\bR^d$ has been extensively studied with work both on when the whole graph is (globally) rigid~\cite{JSS2007,KT2013,JT2022} and when it has a linear sized rigid component~\cite{KMT2011,BLM2018}. Lew, Nevo, Peled, and Raz~\cite{LNPR2022} recently gave a hitting time result: the random graph process in $\bR^d$ becomes globally rigid at exactly the moment it has minimum degree $d + 1$. This implies \cref{thm:hitting} for \emph{generic} embeddings since there are no secure pairs in generic embeddings. However, the restriction to generic embeddings in the definition of global rigidity is a significant weakening. For example, it is folklore (see~\cite[Thm.~63.2.7]{JW2017}) that a graph is globally rigid in $\bR$ if and only if it is 2-connected, while for reconstructing arbitrary point sets, the situation is rather different. Indeed, we show that there are graphs with arbitrarily high connectivity which can be embedded in $\bR$ so that their vertex sets cannot be reconstructed from their edge lengths. This disproves a conjecture of Benjamini and Tzalik~\cite{BT2022}.
\begin{theorem}\label{thm:connectivity}
    Let $k$ be a positive integer. There are $k$-connected graphs on arbitrarily many vertices which can be embedded in the real line in such a way that the largest reconstructible subset of vertices has size $k$.
\end{theorem}

\begin{proof}
    Fix a positive integer $k$ and let $n = \ell k$ be any multiple of $k$. Let $C_1, \dotsc, C_{\ell}$ be $\ell$ disjoint $k$-cliques. Write $C_i = \set{a_{1}^{(i)}, \dotsc, a_{k}^{(i)}}$. Let $G$ be the graph obtained from the union of $C_1, \dotsc, C_{\ell}$ by adding matchings $\set{a_{1}^{(i)} a_{1}^{(i + 1)}, a_{2}^{(i)} a_{2}^{(i + 1)}, \dotsc, a_{k}^{(i)} a_{k}^{(i + 1)}}$ between $C_i$ and $C_{i + 1}$ for each $i = 1, \dotsc, \ell - 1$. Note that $G$ is the Cartesian product of $K_k$ and $P_{\ell}$ and is certainly $k$-connected.
    
    Embed the $C_i$ so that $a_1^{(i)}, \dotsc, a_k^{(i)}$ are $k$ consecutive integers (in that order), that is, $\abs{a_s^{(i)} - a_t^{(i)}} = \abs{s - t}$. We specify that the length of each edge $a_s^{(i)} a_s^{(i + 1)}$ is $k^i$. Certainly such an embedding is possible by placing each $a_s^{(i)}$ at $k^{i - 1} + k^{i - 2} + \dotsb + k + s$. Further, having placed $C_1, \dotsc, C_i$ at these points, there are multiple options for where to place $C_{i + 1}, \dotsc, C_{\ell}$: either to the right of $C_i$ or to the left of $C_1$. In particular, any vertex subset of size at least $k+1$ contains vertices from at least two distinct $C_i$ and is thus not reconstructible.
\end{proof}

The rest of the paper is organized as follows. We establish parts \ref{part:linear} and \ref{part:almostall} of \cref{thm:linear} in \cref{sec:linear}. The remainder of \cref{thm:linear} as well as \cref{thm:hitting,thm:hittingalg} are proved in \cref{sec:hitting}. We conclude in \cref{sec:openproblems} with some discussion of open problems.

Throughout we use standard asymptotic notation and assume that $n$ is sufficiently large.

\section{Reconstructing a linear proportion of the points} \label{sec:linear}

In this section, we prove parts \ref{part:linear} and \ref{part:almostall} of \cref{thm:linear}. Fix a set $V$ of $n$ points on the real line. Our method for reconstructing a linear portion of $V$ is via short cycles. We first show that, for most short cycles $C$ whose vertices are in $V$, if the edges of $C$ were all revealed distances, then the vertices of $C$ can be reconstructed up to isometry. This will follow from the fact that multiple possible embeddings of a short cycle on the real line would yield a small ``linear dependence'' among the pairwise distances of adjacent points on the cycle. We show in \cref{sec:recontuples} that this linear dependence can only occur for a small proportion of short cycles.

Then in \cref{sec:shortcycles} we assume that each pairwise distance is known independently with probability  $p = 42 / n$. We note that a random graph $G \sim \cG(n, 42 / n)$ will typically have $\Omega(n^2)$ pairs of vertices that are contained in a short cycle. Since most short cycles can be reconstructed, we can reconstruct the distances between $\Omega(n^2)$ pairs of points. This allows us to apply the following extremal result of Benjamini and Tzalik~\cite{BT2022} to reconstruct a linear proportion of all points and so prove part~\ref{part:linear}.

\begin{theorem}[Benjamini and Tzalik~\cite{BT2022}]\label{thm:BT}
    Let $V$ be a set of $n$ points on the real line. Suppose that the set of known pairwise distances $\cP$ has size greater than $40n^{3/2}$. Then there is a set of $\Omega(\abs{\cP}/n)$ points that can be reconstructed.
\end{theorem}

Finally, to obtain part~\ref{part:almostall}, we start with this linear proportion of all points that we have already reconstructed. Then, by increasing $p$ and so sprinkling in some additional edges, we will find that for most points we will know at least two distances to the set that we have already reconstructed, and so we can reconstruct those points as well.

\subsection{Reconstructible tuples in embeddings}\label{sec:recontuples}

In this section we show that most short cycles, if their edges are revealed distances, can be reconstructed. To this end we make the following definition.

\begin{definition}
    A tuple $T = (v_1, v_2, \dotsc, v_k)$ of distinct points in $V$ is \defn{cycle-reconstructible} if, given just the pairwise distances $\abs{v_2 - v_1}$, \ldots, $\abs{v_k - v_{k-1}}$, and $\abs{v_1 - v_k}$, we can reconstruct $\{v_1, \dots, v_k\}$ up to isometry.
\end{definition}

Consider some $k$-tuple $T = (v_1, v_2, \dotsc, v_k)$ and let $d_1 = \abs{v_2 - v_1}$, \ldots, $d_{k - 1} = \abs{v_k - v_{k - 1}}$, and $d_k = \abs{v_1 - v_k}$. Certainly there are some $\epsilon_1, \dotsc, \epsilon_{k - 1} \in \set{-1, 1}$ such that
\begin{equation}\label{eq:dk}
    \sum_{i = 1}^{k - 1} \epsilon_i d_i = d_k.
\end{equation}
If $T$ is not cycle-reconstructible, then there are at least two $(\epsilon_1, \dotsc, \epsilon_{k - 1}) \in \set{-1, 1}^{k - 1}$ satisfying \eqref{eq:dk}. Subtracting two such expressions and halving the result shows that there is some non-zero vector $(\gamma_1, \dotsc, \gamma_{k - 1}) \in \set{-1, 0, 1}^{k - 1}$ with
\begin{equation}\label{eq:lindep}
    \sum_{i = 1}^{k - 1} \gamma_i d_i = 0.
\end{equation}
This allows us to bound the number of non-cycle-reconstructible $k$-tuples.

\begin{lemma}\label{lem:unrecon}
    Let $k \leq 0.9 \log n$. At most an $n^{-0.01}$ fraction of all $k$-tuples are not cycle-reconstructible.
\end{lemma}

\begin{proof}
    Select a $k$-tuple $T = (v_1, v_2, \dotsc, v_k)$ of distinct points in $V$ uniformly at random. Let $d_1 = \abs{v_2 - v_1}$, \ldots, $d_{k - 1} = \abs{v_k - v_{k - 1}}$. We say that $v_{s + 1}$ \defn{fails} (for $1 \leq s \leq k - 1$) if there is some non-zero vector $(\gamma_1, \dotsc, \gamma_{s - 1}) \in \set{-1, 0, 1}^{s - 1}$ with $\sum_{i < s} \gamma_i d_i = d_s$. By the preceding discussion, if $T$ is not cycle-reconstructible, then some $v_{s + 1}$ fails.
    
    Given $v_1, \dotsc, v_s$ there are at most $3^{s - 1}$ values of $d_s = \abs{v_{s + 1} - v_s}$ for which there is a vector $(\gamma_1, \dotsc, \gamma_{s - 1}) \in \set{-1, 0, 1}^{s - 1}$ with $d_s = \sum_{i < s} \gamma_i d_i$. Each such $d_s$ corresponds to two possible values of $v_{s + 1}$ (namely $v_s \pm d_s$). In particular, as there are $n - s$ possible choices for $v_{s + 1}$,
    \begin{equation*}
        \bP(v_{s + 1} \textnormal{ fails}) \leq \frac{2 \cdot 3^{s - 1}}{n - s} \leq \frac{3^s}{n},
    \end{equation*}
    and so, taking a union bound,
    \begin{equation*}
        \bP(T \textnormal{ not cycle-reconstructible}) \leq \sum_{s \leq 0.9 \log n} \frac{3^s}{n} \leq \frac{3^{0.9 \log n + 1}}{2n} = \frac{3}{2} n^{0.9 \log 3 - 1} \leq n^{-0.01}. \qedhere
    \end{equation*}
\end{proof}

We say that a pair of points $u, v$ is \defn{$k$-bad} if greater than an $n^{-0.005}$ fraction of the $k$-tuples containing both $u$ and $v$ are not cycle-reconstructible. Given the preceding lemma, it follows that there are only few $k$-bad pairs of points.

\begin{lemma}\label{lem:bad}
    Let $k \leq 0.9 \log n$. At most an $n^{-0.005}$ fraction of all pairs are $k$-bad.
\end{lemma}

\begin{proof}
    Note that every pair of points is in the same number of $k$-tuples -- call this common value $A$. Let $u, v$ be a uniformly random pair of points in $V$ and let $X$ be the number of non-cycle-reconstructible $k$-tuples containing both $u$ and $v$. We will double count the number of $(T, \set{u, v})$ where $T$ is a non-cycle-reconstructible $k$-tuple and the pair $u, v$ is contained in $T$. By \cref{lem:unrecon}, this number is at most 
    \begin{equation*}
        n^{-0.01} \cdot \text{number of $k$-tuples} \cdot \tbinom{k}{2} = n^{-0.01} \cdot A \tbinom{n}{2}.
    \end{equation*}
    On the other hand, the number is exactly equals $\bE(X) \cdot \binom{n}{2}$. Hence $\bE(X) \leq An^{-0.01}$. Thus, by Markov's inequality,
    \begin{equation*}
        \bP(u, v \textnormal{ is $k$-bad}) = \bP(X > A n^{-0.005}) \leq \frac{\bE(X)}{A n^{-0.005}} \leq n^{-0.005}. \qedhere
    \end{equation*}  
\end{proof}

If a pair of points $u, v$ is not $k$-bad for any $k \le 0.9 \log n$, then a random short cycle in $\cG(n,p)$ containing $u$ and $v$ is likely to allow us to reconstruct the distance between $u$ and $v$. This motivates the following definition.

\begin{definition}
    A pair of points $u, v$ is \defn{useless} if there is some $k \leq 0.9 \log n$ such that the pair $u, v$ is $k$-bad. Otherwise $u, v$ is \defn{useful}.
\end{definition}

Using the previous results, we can bound the number of useless pairs of points.

\begin{corollary}\label{cor:useless}
    There are at most $n^{2 - 0.004}$ useless pairs of points.
\end{corollary}

\begin{proof}
    By \cref{lem:bad}, the number of useless pairs of points is at most
    \begin{equation*}
        \sum_{k \leq 0.9 \log n} \textnormal{number of $k$-bad pairs} \leq (0.9 \log n) n^{2 - 0.005} \leq n^{2 - 0.004}. \qedhere
    \end{equation*}
\end{proof}

\subsection{Reconstructing from short cycles in \texorpdfstring{$\cG(n, p)$}{Gnp}}\label{sec:shortcycles}

Given the results from the previous section, it will be helpful to show that most pairs of points in $\cG(n,p)$ are contained in a short cycle. We will use the following straightforward result about random graphs.

\begin{lemma}\label{lem:shortcycles}
    A random graph $G \sim \cG(n, 42/n)$ contains whp $\Omega(n^2)$ pairs of vertices that are in a cycle of length at most $0.9 \log n$.
\end{lemma}

Let $\diam(G)$ denote the largest distance between two vertices in the same component of $G$. We will use the following consequence of a (much more general) result by Riordan and Wormald~\cite[Thm.~1.1]{RW2010}. There are alternative elementary arguments that prove results similar to Lemma \ref{lem:shortcycles}, but we use this result for brevity.

\begin{lemma}\label{lem:diam}
A random graph $G \sim \cG(n, 21/n)$ satisfies whp $\diam(G)\leq 0.44 \log n$.
\end{lemma}

\begin{proof}[Proof of Lemma~\ref{lem:shortcycles}]
Let $V = V_1 \cup V_2$ with $\abs{V_1} = \abs{V_2} = n/2$. First reveal the edges inside $V_1$ and $V_2$, and let $C_1$ and $C_2$ be the largest components in $G[V_1]$ and $G[V_2]$ respectively. $G[V_1]$ and $G[V_2]$ are independently distributed as $\cG(n/2, 21/(n/2))$. 

By standard facts about random graphs (see, for example, \cite[Ch.~6]{Bollobas}), $C_1$ and $C_2$ both have size $\Omega(n)$ whp. Also, \cref{lem:diam} implies that whp $C_1$ and $C_2$ have diameter at most $0.44 \log n$.

Next, we reveal the edges between $V_1$ and $V_2$.  For a fixed vertex $v \in C_1$ write $X_v$ for the indicator random variable of the event that $v$ has a neighbour in $C_2$. We have $\bP(X_v = 0) = (1 - 42/n)^{\abs{C_2}} \leq e^{-42\abs{C_2}/n}$ which is at most some constant $a < 1$. Since the variables $X_v$ for $v \in C_1$ are independent, by Chernoff, we have whp that $\sum_v X_v \geq (1 - a) \abs{C_1}/2$, and so the set $A$ of vertices in $C_1$ that have a neighbour in $C_2$ has size at least $\Omega(n)$. 
Now, for any pair of vertices $u,v \in A$, concatenating the shortest path between them in $C_1$, their respective edges to vertices $w_u,w_v\in C_2$, and the shortest path between $w_u$ and $w_v$ in $C_2$ gives a cycle containing $u$ and $v$ of length at most $0.44 \log n + 0.44 \log n + 2 \leq 0.9 \log n$.
\end{proof}

Recall that $\cP$ is the set of known distances. We say that a pair of points $u, v$ is \defn{close} if $u$ and $v$ are in a cycle in $(V, \cP)$ of length at most $0.9 \log n$. We say that the pair $u, v$ is \defn{deducible} if $\abs{u - v}$ can be uniquely determined from $\cP$. Note that if some cycle-reconstructible cycle in $\cP$ contains both $u$ and $v$, then the pair $u, v$ is deducible. With this in mind, we bound the number of useful pairs of points that are in a short cycle, but where we cannot determine their distance.

\begin{lemma}\label{lem:nonded}
    If $(V, \cP)$ is distributed as $\cG(n,p)$, then whp there are at most $n^{2 - 0.002}$ pairs of points that are close and useful but not deducible.
\end{lemma}

\begin{proof}
    Fix a pair of points $u, v \in V$ and let $k \leq 0.9 \log n$. Conditioned on the event that $(V, \cP)$ contains a $k$-cycle which contains $u$ and $v$, the probability that any particular $k$-tuple containing $u$ and $v$ appears in $(V, \cP)$ as a cycle is uniform. If $u, v$ is useful, then at most an $n^{-0.005}$ fraction of these $k$-tuples is non-reconstructible, and so
    \begin{equation*}
        \bP(u, v \textnormal{ not deducible} \mid u, v \textnormal{ in a $k$-cycle and useful}) \leq n^{-0.005}.
    \end{equation*}
    This holds for all $k \leq 0.9 \log n$ and so $\bP(u, v \textnormal{ not deducible} \mid u, v \textnormal{ close and useful}) \leq n^{-0.005}$. Hence,
    \begin{equation*}
        \bP(u, v \textnormal{ close and useful but not deducible}) \leq n^{-0.005}.
    \end{equation*}
    Let $X$ denote the number of pairs which are close and useful but not deducible. Then $\bE(X) \le n^{-0.005} \cdot \binom{n}{2}$. By Markov's inequality,
    \begin{equation*}
        \bP(X > n^{2 - 0.002}) \le \frac{\bE(X)}{n^{2 - 0.002}} \le n^{-0.003}. \qedhere
    \end{equation*}
\end{proof}

We are now ready to prove the statements \ref{part:linear} and \ref{part:almostall} of \cref{thm:linear}.

\begin{proof}[Proof of \cref{thm:linear}\ref{part:linear}]
    By \cref{lem:shortcycles}, whp there are $\Omega(n^2)$ close pairs of points in $(V, \cP)$. Of the close pairs, at most $n^{2 - 0.004}$ pairs are useless by \cref{cor:useless} and whp at most $n^{2 - 0.002}$ are useful but not deducible by \cref{lem:nonded}. In particular, whp at least $\Omega(n^2)$ pairs of vertices in $(V, \cP)$ are deducible.
    
    Finally, \cref{thm:BT} implies that whp there is a set of $\Omega(n)$ vertices that can be reconstructed up to isometry.
\end{proof}

Note that if we have reconstructed a set $S$ of points and we know at least two distances from another point $v$ to some points in $S$, this uniquely determines the position of $v$ relative to $S$, and so we can reconstruct $S \cup \set{v}$. This allows us to prove part~\ref{part:almostall}.

\begin{proof}[Proof of \cref{thm:linear}\ref{part:almostall}]
    By \cref{thm:linear}\ref{part:linear}, there is a constant $c > 0$ such that for $p = 42 / n$ whp there is a set $R$ of at least $c n$ points that can be reconstructed. Let $d>c$ be a constant, $q = d / (c n)$ and consider $G \sim \cG(n, q)$. Every vertex $v \notin R$ satisfies
    \begin{equation*}
        \bP(\abs{N_{G}(v) \cap R} < 2) = (1 - q)^{\abs{R}} + \abs{R} q (1 - q)^{\abs{R} - 1} \le 2 n q (1 - q)^{c n} \le (2 d / c) e^{-d}.
    \end{equation*}
    Since all of these events are independent, a Chernoff bound implies that whp at most $(4 d / c) e^{-d} n$ points $v \notin R$ have at most one neighbour in $R$. Because every point with two neighbours in $R$ can also be reconstructed relative to $R$, it follows that whp a set of at least $(1 - (4 d / c) e^{-d}) n$ points can be reconstructed in $\cG(n, p) \cup \cG(n, q)$ and therefore also in $\cG(n, p + q) \sim \cG(n, \omega / n)$ where $\omega = 42 + d / c$. Since $(4 d / c) e^{-d} \to 0$ when $d \to \infty$, this proves the claim.
\end{proof}

\section{Hitting time for complete reconstruction}\label{sec:hitting}

In this section, we prove \cref{thm:hitting}. For this, we use the following standard coupling of random graphs. Let $(U_e)_{e \in E(K_n)}$ be a collection of independent random variables all uniformly distributed on $[0, 1]$. For $p \in [0, 1]$, let \defn{$G_p$} be the subgraph of $K_n$ whose edges are exactly those edges $e$ with $U_e \leq p$. Note that $G_p$ is distributed as $\cG(n, p)$ and, as $p$ is increased continuously from 0 to 1, $G_p$ evolves as a random graph process where edges are added one-by-one in a uniformly random order (note that with probability 1 all $U_e$ are distinct). We use this to model the process of revealing distances between pairs of points in $V$ one-by-one uniformly at random. We will freely use the following facts about $\cG(n, p)$ in this section (see, for example, \cite[Ch.~7]{Bollobas}). For any positive integer $k$: if $p \leq (\log n + (k - 1) \log\log n - \omega(1))/n$, then whp $\cG(n, p)$ has minimum degree at most $k - 1$ and if $p \geq (\log n + (k - 1) \log\log n + \omega(1))/n$, then whp $\cG(n, p)$ has minimum degree at least $k$.

Three particular values of $p$ will be important for our analysis. \Cref{thm:linear}\ref{part:almostall} implies that there is some large constant $\omega$ such that whp there is a set $R \subseteq V$ of size at least $0.9n$ that can be reconstructed in $G_{\omega/n}$. Let $p_1 = \omega/n$. Let $p_2 = p_1 + (0.9 \log n)/n$ and $p_3 = (2 \log n)/n$. Note that whp $G_{p_2}$ still contains isolated vertices and so the whole of $V$ is not yet reconstructible. We will show that the whole of $V$ is reconstructible in $G_{p_3}$ and so our analysis will focus on the range $[p_2, p_3]$. 

\begin{lemma}\label{lem:gp123}
    Whp, for every edge $uv$ in $G_{p_3}$ with $u, v \notin R$, at least one of $u, v$ has at least two neighbours among $R$ in $G_{p_2}$.
\end{lemma}

\begin{proof}
    First, reveal all edges in $G_{p_1}$. This determines $R$ and a Chernoff bound implies that whp $G_{p_1}$ has at most $2 p_1 \tbinom{n}{2} \le \omega n$ edges. Now, reveal the edges of $G_{p_3}$ in the complement of $R$. This adds every edge independently with a probability of at most $p_3$ to the graph $G_{p_1}$. Hence, a Chernoff bound implies that whp this adds at most $2 p_3 \tbinom{n}{2} \le 2 n \log n$ edges to the complement of $R$, and so $G_{p_3}$ has in total at most $\omega n + 2 n \log n \le 3 n \log n \eqqcolon m$ edges in the complement of $R$.

    Finally, reveal the edges of $G_{p_2}$ between $R$ and the complement of $R$. Each such edge was already present in $G_{p_1}$ or is added independently with probability at least $p \coloneqq p_2 - p_1 = (0.9 \log n) / n$ to the graph $G_{p_1}$. Therefore, every point $v \notin R$ satisfies
    \begin{align*}
        \bP(\abs{N_{G_{p_2}}(v) \cap R} < 2) & \le (1 - p)^{\abs{R}} + \abs{R} p (1 - p)^{\abs{R} - 1} \le 2 n p (1 - p)^{\abs{R}} \\
        & \le 2 n p (1 - p)^{0.9 n} \le (2 n p) e^{- 0.9 p n} \le (2 \log n) n^{-0.81} \le n^{-0.8}.
    \end{align*}
    Since these events are independent between distinct points $u, v \notin R$, this implies
    \begin{equation*}
        \bP(\max(\abs{N_{G_{p_2}}(u) \cap R}, \abs{N_{G_{p_2}}(v) \cap R}) < 2) \le (n^{-0.8})^2 = n^{-1.6}.
    \end{equation*}
    Let $X$ denote the number of edges $u v$ in $G_{p_3}$ such that $u, v \notin R$, but $u$ and $v$ both have at most one neighbour among $R$ in $G_{p_2}$. Since $G_{p_3}$ has at most $m$ edges in the complement of $R$, this implies that $\bE(X) \le m \cdot n^{-1.6} \le 3 n^{-0.6} \log n \le n^{-0.5}$, and so $\bP(X \ge 1) \le n^{-0.5}$.
\end{proof}

If a vertex $v$ has two neighbours among a reconstructible set, then the position of $v$ can be reconstructed with respect to that set. With this in mind, we define, for each $p \in [0, 1]$, \defn{$R'_p$} and \defn{$R''_p$} as follows.
\begin{align*}
    R'_p & = R \cup \set{v \in V \colon \abs{N_{G_p}(v) \cap R} \geq 2}, \\
    R''_p & = R'_p \cup \set{v \in V \colon \abs{N_{G_p}(v) \cap R'_p} \geq 2}.
\end{align*}
We now collect some important facts about the $R'_p$ and $R''_p$.
\begin{lemma}\label{lem:Rp}
    With high probability the following all hold.
    \begin{enumerate}[label = \alph{*}., ref = \alph{*}]
        \item For all $p \geq p_1$, $R'_p$ and $R''_p$ are reconstructible in $G_p$. \label{part:recon}
        \item For all $p \in [p_2, p_3]$, every edge in $G_p$ is incident to a vertex in $R'_p$. \label{part:incident}
        \item For all $p \in [p_2, p_3]$, $R''_p$ contains every vertex of degree at least two in $G_p$. \label{part:deg2}
    \end{enumerate}
\end{lemma}

\begin{proof}
    Note that $R$ is reconstructible in $G_{p_1}$. For any $p \geq p_1$, $G_p$ contains $G_{p_1}$ and so $R$ is reconstructible in $G_p$. $R'_p$ consists of $R$ and all vertices with at least two neighbours in $R$ and so is reconstructible in $G_p$. Similarly for $R''_p$. This proves \ref{part:recon}.
    
    By \cref{lem:gp123}, whp every edge in $G_{p_3}$ has an end-point in $R'_{p_2}$. Now let $p \in [p_2, p_3]$ and let $uv$ be an edge of $G_p$. Then $uv$ is an edge of $G_{p_3}$. Hence, at least one of $u$, $v$ is in $R'_{p_2} \subseteq R'_p$. This proves \ref{part:incident}.

    Let $u$ be a vertex of degree at least two in $G_p$. By \ref{part:incident}, either $u$ is in $R'_p$ or every neighbour of $u$ is in $R'_p$. Either way, $u$ is in $R''_p$ which proves \ref{part:deg2}.
\end{proof}

Now whp $\cG(n, p_3)$ has minimum degree at least two and so, by the time that the random graph process from the coupling reaches minimum degree two, \cref{lem:Rp} tells us that whp the whole of $V$ can be reconstructed. \Cref{thm:linear}\ref{part:all} follows immediately.

We say that an ordered pair of points $(u, v)$ is \defn{uncertain} in $G_p$ if $u v$ is the only edge incident to $u$ and the point $w = 2 v - u$ is a point of $V$ with degree one in $G_p$. We call $u$ uncertain in $G_p$, if $(u, v)$ is uncertain in $G_p$ for some $v \in V$.

\begin{lemma}\label{lem:uncertain}
    Whp, for all $p \in [p_2, p_3]$, $G_p$ contains no uncertain point.
\end{lemma}

\begin{proof}
    Let $(u, v)$ be an ordered pair of points such that $w = 2 v - u$ is a point of $V$. If $(u, v)$ is uncertain in $G_p$ for some $p \in [p_2, p_3]$, then there is a point $x$ such that the edges $u v$ and $w x$ are present in $G_p \subseteq G_{p_3}$, but every other edge incident to $u$ and $w$ cannot be present in $G_p \supseteq G_{p_2}$. In particular, we get that
    \begin{align*}
        \bP((u, v) \text{ is uncertain in some } G_p) & \le n p_3^2 (1 - p_2)^{2 n - 5} \le 2 n p_3^2 (1 - p_2)^{2 n} \\
        & \le (2 n p_3^2) e^{- 2 n p_2} \le 8 (\log n)^2 n^{-1} n^{-1.8} \le n^{-2.5}.
    \end{align*}
    Let $X$ denote the number of ordered pairs of points $(u, v)$ which are uncertain in some $G_p$. Since there are at most $n^2$ ordered pairs of points, this implies that $\bE(X) \le n^2 \cdot n^{-2.5} = n^{-0.5}$, and so $\bP(X \ge 1) \le n^{-0.5}$.
\end{proof}

We say that a point $u$ is \defn{undecidable} in $G_p$ if it has degree one in $G_p$ and $(u, v)$ is not secure where $v$ is the unique neighbour of $u$. Note that \cref{thm:hitting} says that whp $G_p$ is reconstructible exactly when it has no isolated nor undecidable points.

\begin{proof}[Proof of \cref{thm:hitting}]
    Using \cref{lem:Rp} and \cref{lem:uncertain}, we know that whp all of the following hold for all $p \in [p_2, p_3]$. Firstly, $R''_p$ can be reconstructed in $G_p$. Secondly, every edge in $G_p$ is incident to a vertex in $R''_p$. Thirdly, $G_p$ contains no uncertain point. Fourthly, $R''_p$ contains all vertices of degree at least two in $G_p$. Finally, $G_{p_2}$ has isolated points while $G_{p_3}$ has minimum degree at least two.

    Now, consider the minimal $p$ such that the graph $G_p$ has no isolated and no undecidable points. By the final property, we know that $p_2 < p \le p_3$. In particular, $R''_p$ can be reconstructed in $G_p$ by the first property. Let $u$ be a point in the complement of $R''_p$. Since $u$ cannot be an isolated point, $u$ must have a unique neighbour $v$, and that neighbour must be in $R''_p$ by the second property. So $v$ has been reconstructed and we therefore know that $u$ can only be at position $u$ or $w = 2 v - u$.

    Given that $u$ is not undecidable, $(u, v)$ must be secure which means that $w$ is a point of $V$. Since $w$ cannot be isolated and $u$ is not uncertain by the third property, it then follows that the degree of $w$ must be at least two. Hence, by the fourth property, $w \in R''_p$, and so $w$ has been reconstructed. So we know that the position $w$ is already occupied by a point different from $u$, and so $u$ can only be at a single possible position relative to the rest of the points. Hence, we can also reconstruct $u$. Because $u$ was an arbitrary point of the complement of $R''_p$, it follows that all points of $G_p$ can be reconstructed up to isometry.

    Finally note that the graph preceding $G_p$ in the random graph process has an isolated or undecidable point $v \in V$. Such a point has at least two possible positions relative to the rest of $V$ and is therefore not determined up to isometry. Hence, $G_p$ is the first graph in the random graph process that can be reconstructed up to isometry.
\end{proof}

\cref{thm:hittingalg} is now an easy consequence as we can simply follow the proofs from this section and the previous section to obtain the algorithm that we want.

\begin{proof}[Proof of \cref{thm:hittingalg}]
    Our algorithm is composed of the following steps.
    \begin{enumerate}[label = \alph{*}., ref = \alph{*}]
        \item Reveal pairwise distances one-by-one until the graph of known pairwise distances has $42n$ edges. Call this graph $G'$.
        \item Enumerate all walks (that do not repeat edges) of length at most $0.9 \log n$ in $G'$ by starting at an arbitrary vertex and picking neighbours sequentially. From this, find a list of all cycles of length at most $0.9 \log n$ in $G'$.\label{part:findcycles}
        \item Then check each cycle $v_1 \dots v_k$ of length at most $0.9 \log n$ in $G'$. If the distances between consecutive points on the cycle are $d_1, \dots, d_k$ and there are unique $\epsilon_1, \dots, \epsilon_{k - 1} \in \set{-1, 1}$ such that $\sum_{i = 1}^{k - 1} \epsilon_i d_i = d_k$, then the distance between $v_i$ and $v_j$ must be $\abs{\sum_{l = i}^{j - 1} \epsilon_l d_l}$.\label{part:checkcycles}
        \item Using these new distances, apply \cref{thm:BT} to reconstruct a set of points $R$ (a polynomial time algorithm can be obtained from the proofs in \cite{BT2022}).
        \item Reconstruct and add to $R$ every vertex with two neighbours in $R$ as well as every vertex with one neighbour in $R$ whose other possible position is already occupied by a vertex of $R$.
        \item Continue to reveal pairwise distances one-by-one. Each time a distance between points $u, v$ is added where $v \in R$ and $u \notin R$, check if $u$ has two neighbours in $R$ or if the reflection of $u$ over $v$ is in $R$. If either of these two outcomes occur, then reconstruct $u$ and add $u$ to $R$.
    \end{enumerate}
    Let $p$ be such that $G'$ corresponds to $G_p$ in the coupling. By a Chernoff bound, with an exponential (in $n$) failure probability, $42/n < p < 168/n$. When $p < 168/n$, the expected number of walks in $G_p$ (that do not repeat edges) of length at most $0.9 \log n$ is at most
    \begin{equation*}
        \sum_{k \leq 0.9 \log n} p^k n^{k + 1} < n \sum_{k \leq 0.9 \log n} 168^k < n^6,
    \end{equation*}
    and the number of such walks is always at most $n^{0.9 \log n}$. Hence, the expected number of such walks in $G'$ is polynomial in $n$ which implies that the expected runtime of step~\ref{part:findcycles} is polynomial in $n$. It also follows that if $X$ is the number of cycles of length at most $0.9 \log n$ in $G'$, then $\bE(X)$ is polynomial in $n$. For each cycle of length at most $0.9 \log n$, there are at most $2^{0.9 \log n}$ possible $(\epsilon_1, \dotsc, \epsilon_{k - 1})$ and so each check in step~\ref{part:checkcycles} takes time polynomial in $n$, implying that the runtime of part~\ref{part:checkcycles} is $X$ times a polynomial in $n$. Hence, the expected runtime of the first three steps is polynomial in $n$. The final three steps can be performed in polynomial time.

    The set $R$ obtained by the algorithm is certainly reconstructible. We are left to check that whp $R = V$ exactly when $V$ is reconstructible. As shown in the proofs of \cref{thm:linear} and \cref{thm:hitting}, whp the initial set $R$ will have size $\Omega(n)$, and we will then add every point to $R$ that has degree two or that has degree one and is not undecidable. Thus, by \cref{thm:hitting}, whp this algorithm will reconstruct $V$ exactly at the time when this is possible.
\end{proof}

We have already shown that if every point is incident to at least two known distances, then whp the whole of $V$ is reconstructible. While in some cases the whole of $V$ might be reconstructible even before that time, we now show that this does not happen much earlier.

\begin{theorem}\label{thm:arbitrarily}
    Let $f \colon \bN \to \bR$ be any function going to infinity. Let $V$ be a set of $n$ points on the real line. Suppose the distances between pairs of points in $V$ are revealed one-by-one in a uniformly random order. If there are at least $f(n)$ points that are incident to only one revealed distance, then whp it is not possible to reconstruct the whole of $V$.
\end{theorem}

\begin{proof}
    We may and will assume that $f(n) \leq n/100$ for all $n$. Let $G = (V, \cP)$ be the graph of known distances and suppose $G$ has at least $f(n)$ points of degree one. It suffices to show that whp there is a point $u$ whose unique edge $uv$ is such that $(u, v)$ is not secure.
    
    We first give an upper bound on the number of secure pairs of points. Let the points in $V$ be $v_1 < v_2 < \dotsb < v_n$. Fix $k \leq n/2$. There are at most $k - 1$ points $u \in \set{v_{k + 1}, \dotsc, v_n}$ for which $(u, v_k)$ is secure as the reflection of $u$ over $v_k$ must be one of $v_1, \dotsc, v_{k - 1}$. In particular, there are at most $2(k - 1)$ points $u \in V$ for which $(u, v_k)$ is secure, and so the number of secure pairs whose second vertex is in the first half of $V$ is
    \begin{equation*}
        \sum_{k \leq n/2} 2(k - 1) = 2\binom{\floor{n/2}}{2} \leq n^2/4.
    \end{equation*}
    Hence, the total number of secure pairs is at most $n^2/2$.

    Now, perform the following random process. First pick a uniformly random degree-one vertex $u_1$ of $G$ and let $v_1$ be the unique neighbour of $u_1$. For $s = 2, \dotsc, \floor{f(n)/3}$ let $u_s$ be a uniformly random degree-one vertex in $V \setminus \set{u_1, v_1, u_2, v_2, \dotsc, u_{s - 1}, v_{s - 1}}$ and let $v_s$ be the unique neighbour of $u_s$. Given $\set{u_1, v_1, u_2, v_2, \dotsc, u_{s - 1}, v_{s - 1}}$, note that $(u_s, v_s)$ is a uniformly random ordered pair from
    \begin{equation*}
        \set{(x, y) \colon x \neq y,\ x \in V \setminus \set{u_1, v_1, \dotsc, u_{s - 1}, v_{s - 1}},\ y \in V \setminus \set{u_1, u_2 \dotsc, u_{s - 1}}}.
    \end{equation*}
    This set has size at least $3n^2/4$ and so, given $\set{u_1, v_1, u_2, v_2, \dotsc, u_{s - 1}, v_{s - 1}}$, the probability that $(u_s, v_s)$ is secure is at most $(n^2/2)/(3n^2/4) = 2/3$. In particular, the probability that $(u_s, v_s)$ is secure for all $s \leq \floor{f(n)/3}$ is at most $(2/3)^{\floor{f(n)/3}} = o(1)$.
\end{proof}

Now, if $p \leq (\log n + \log\log n - \omega(1))/n$, then whp $\cG(n, p)$ has arbitrarily many degree-one vertices (see, for example, \cite[Ch.~7]{Bollobas}) and so \cref{thm:linear}\ref{part:notall} follows.

\section{Open problems}\label{sec:openproblems}

\Cref{thm:linear} shows that $1/n$ is a weak threshold for a linear sized reconstructible set of vertices. It would be interesting to determine whether there is a sharp threshold and it is natural to conjecture that this occurs at $1/n$ as this coincides with the appearance of the giant component in $\cG(n, p)$.

\begin{conjecture}
    Let $V$ be a set of $n$ points on the real line. Suppose the graph of known pairwise distances $(V, \cP)$ is distributed as $\cG(n, (1 + \epsilon)/n)$ where $\epsilon > 0$ is a constant. Then whp there is a reconstructible set of size $\Omega_{\epsilon}(n)$.
\end{conjecture}

We have not tried to optimize the constant in \cref{thm:linear}\ref{part:linear}, but our approach will not give a bound better than $9/n$, as for smaller $p$ there are too few pairs in sufficiently short cycles. A bootstrapping argument might push the constant below $9$, but would not get down to $p = (1 + \epsilon)/n$, as most pairs will only be in cycles of length at least $\Omega(\log(n)/\epsilon)$.

It would also be interesting to give a characterisation of the giant reconstructible component. For global graph rigidity in one-dimension the largest reconstructible component is the largest 2-connected subgraph. For graph rigidity in two-dimensions the threshold for the emergence of a giant reconstructible component was determined in~\cite{KMT2011} and a characterisation of this component was determined in~\cite{BLM2018}.

Reconstructibility in dimensions greater than one is also interesting. Reconstructing the whole of $V$ is too much to ask: consider an embedding where $n - 2$ points lie in a $(d - 1)$-dimensional hyperplane and the other two points $u, v$ do not. $V$ can only be fully reconstructed if the distance between $u$ and $v$ is revealed (otherwise $u, v$ could be on the same side of the hyperplane or on opposite sides). However, it is interesting to ask for the threshold at which a linear sized subset of $V$ can be reconstructed. Is $1/n$ a weak threshold as it is for $d = 1$?

Finally, Benjamini and Tzalik~\cite{BT2022} also considered an even stronger notion of reconstructibility. Let $G = (V, E)$ be a graph. We say a subset $U \subseteq V$ is \defn{adversarially reconstructible in $\bR$} if, for \emph{every} embedding of $V$ in $\bR$, $U$ is reconstructible from the distances $\abs{u - v}$  for $uv \in E$. Note this is similar to the definition of global rigidity with the generic condition removed. It would be interesting to determine the thresholds for $\cG(n, p)$ to be adversarially reconstructible in $\bR$ and for some linear sized subset of $\cG(n, p)$ to be adversarially reconstructible. In contrast to \cref{thm:linear}, minimum degree at least two is necessary as the embedding can be chosen so there are no secure pairs. Benjamini and Tzalik conjectured it is also sufficient: when distances are revealed one-by-one in a random order, the graph becomes adversarially reconstructible exactly at the first time that it has minimum degree two.

{
\fontsize{11pt}{12pt}
\selectfont
	
\hypersetup{linkcolor={red!70!black}}
\setlength{\parskip}{2pt plus 0.3ex minus 0.3ex}

}

\end{document}